\documentclass[11pt]{amsart}
\usepackage{amssymb,pstricks,pst-plot}
\input xy
\xyoption{all}

\definecolor{OrangeRed}{cmyk}{0,0.6,1,0}
\definecolor{DarkBlue}{cmyk}{1,1,0,0.20}
\definecolor{Black}{cmyk}{0,0,0,1}
\definecolor{Violet}{cmyk}{0.79,0.88,0,0}
\definecolor{myblue}{rgb}{0.66,0.78,1.00}

\parskip=\smallskipamount

\newtheorem{theorem}{Theorem}[section]
\newtheorem{lemma}[theorem]{Lemma}

\theoremstyle{definition}

\newtheorem{problem}[theorem]{Problem}

\newcommand{\C}{\mathbb{C}}
\newcommand{\D}{\mathbb{D}}
\newcommand{\N}{\mathbb{N}}

\newcommand{\cC}{\mathcal{C}}

\newcommand{\cO}{\mathcal{O}}

\newcommand{\ggoth}{{\ensuremath{\mathfrak{g}}}}

\newcommand\Id{\mathrm{Id}}
\newcommand\bihol{\stackrel{\cong }{\longrightarrow}}
\newcommand{\Aut}{\mathrm{Aut}\,}
\newcommand\dist{\mathrm{dist}}
\newcommand\hra{\hookrightarrow}
 
\newcommand\wh{\widehat}

\newcommand\wt{\widetilde}

\newcommand{\Lie}{\mathrm{Lie}}

\def\bs{\backslash}

\numberwithin{equation}{section}

\begin{document}
\title[Holomorphic families of long $\C^2$'s]
{Holomorphic families of long $\C^2$'s}
\author{Franc Forstneri\v c}
\address{Faculty of Mathematics and Physics, University of Ljubljana, 
and Institute of Mathematics, Physics and Mechanics, Jadranska 19, 
1000 Ljubljana, Slovenia}
\email{franc.forstneric@fmf.uni-lj.si}
\thanks{Supported by grants P1-0291 and J1-2152 from ARRS, Republic of Slovenia.}

%
%
%
%
\date{\today}

%
%
%
%
\begin{abstract}
We construct a holomorphically varying family
of complex surfaces $X_s$, parametrized by 
the points $s$ in any Stein manifold, such that
every $X_s$ is a long $\C^2$ which is biholomorphic 
to $\C^2$ for some but not all values of $s$.
\end{abstract}
\maketitle

\section{The main result}
\label{LongC2}
A complex manifold $X$ of dimension $n$ is a 
{\em long $\C^n$} if $X=\cup_{j=1}^\infty X^j$ 
where $X^1\subset X^2\subset X^3\subset \ldots$ is 
an increasing sequence of open domains exhausting $X$
such that each $X^j$ is biholomorphic to $\C^n$.
Every long $\C$ is biholomorphic to $\C$.
However, for every $n>1$ there exists a long $\C^n$ which 
is not a Stein manifold, and in particular is 
not biholomorphic to $\C^n$. Such manifolds have been 
constructed recently by E.\ F.\ Wold \cite{Wold2010} 
by using his example of a non-Runge Fatou-Bieberbach domain 
in $\C^2$ \cite{Wold2008}, thereby solving a problem 
posed by J.\ E.\ Forn\ae ss  \cite{Forn2004}.
Let us also mention that for every $n\ge 3$ 
Forn\ae ss \cite{Forn1976} constructed an $n$-dimensional
non-Stein complex manifold that is exhausted by biholomorphic images 
of the polydisc; his construction uses Wermer's example of a 
non-Runge embedded polydisc in $\C^3$ \cite{Wermer}.

Recently L.\ Meersseman asked (private communication)
whether it is possible to holomorphically deform the standard $\C^n$
to a long $\C^n$ that is not biholomorphic to $\C^n$.
Here we give a positive answer and show that the behavior of long
$\C^n$'s in a holomorphic family can be rather chaotic.

\begin{theorem}
\label{t:main}
Fix an integer $n>1$.
Assume that $S$ is a Stein manifold, $A=\cup_{j} A_j$ is a 
finite or countable union of closed complex subvarieties of $S$, 
and $B=\{b_j\}$ is a countable set in $S\bs A$. 
Then there exists a complex manifold $X$ and a 
holomorphic submersion $\pi\colon X\to S$ onto $S$ such that
\begin{itemize}
\item[\rm (i)] the fiber $X_s=\pi^{-1}(s)$ is a long $\C^n$
for every $s\in S$,
\item[\rm (ii)]  $X_s$ is biholomorphic to $\C^n$ for every $s\in A$, and
\item[\rm (iii)] $X_s$ is non-Stein for every $s\in B$.
\end{itemize}
\end{theorem}

In particular, for any two disjoint countable sets $A,B\subset \C$ 
there is a holomorphic family $\{X_s\}_{s\in \C}$ of long $\C^2$'s 
such that $X_s$ is biholomorphic to $\C^2$ for all $s\in A$ and is 
non-Stein for all $s\in B$. This is particularly striking if the sets 
$A$ and $B$ are chosen to be everywhere dense in $\C$.

The conclusion of Theorem \ref{t:main} can be strengthened 
by adding to the set $B$ a closed complex subvariety 
of $X$ contained in $X\bs A$. We do not know whether the same 
holds if $B$ is a countable union of subvarieties of $X$.

Several natural questions appear:

\begin{problem}
Given a holomorphic family $\{X_s\}_{s\in S}$ of long $\C^n$'s
for some $n>1$, what can be said about the set of points $s\in S$ for which 
the fiber $X_s$ is (or is not) biholomorphic to $\C^n$?
Are these sets necessarily a $G_\delta$, an $F_\sigma$,
of the first, resp.\ of the second category, etc.?
\end{problem}

A more ambitious project would be to answer the following question:

\begin{problem}
Is there a holomorphic family $X_s$ of long $\C^2$'s,
parametrized by the disc $\D=\{s\in\C\colon |s|<1\}$ 
or the plane $\C$, such that $X_s$ is not biholomorphic 
to $X_{s'}$ whenever $s\ne s'$?
\end{problem}

We do not know of any effective criteria 
to distinguish two long $\C^n$'s from each other,
except if one of them is the standard $\C^n$
and the other one is non-Stein. 
Apparently there is no known
example of a Stein long $\C^n$ other than $\C^n$.
It is easily seen that any two long $\C^n$'s are 
smoothly diffeomorphic to each other, so the gauge-theoretic
methods do not apply. 

To prove Theorem \ref{t:main} we follow Wold's construction 
of a non-Stein long $\C^2$ \cite{Wold2010}, but doing all key 
steps with families of Fatou-Bieberbach maps depending holomorphically
on the parameter in a given Stein manifold $S$.
(The same proof applies for any $n\ge 2$.)
By using the And\'ersen-Lempert theory
\cite{AL,FRosay,Varolin2001,Varolin2000}
we insure that in a holomorphically varying
family of injective holomorphic maps $\phi_s \colon \C^2\hra \C^2$ 
$(s\in S)$ the image domain $\phi_s(\C^2)$ is Runge for 
some but not all values of the parameter. 
In the limit manifold $X$ we thus get fibers $X_s$ 
that are biholomorphic to $\C^2$, as well as fibers that are
not holomorphically convex, and hence non-Stein.

\section{Constructing holomorphic families of long $\C^n$'s}
\label{S2}
Let $S$ be a complex manifold 
that will be used as the parameter space. 
We recall how one constructs a complex manifold $X$ and a holomorphic
submersion $\pi\colon X\to S$ such that the fiber 
$X_s=\pi^{-1}(s)$ is a long $\C^n$ for each $s\in S$. 
(This is a parametric version of the construction in 
\cite{Forn1976} or \cite[\S 2]{Wold2010}.)

Assume that we have a sequence of injective holomorphic maps 
\begin{equation}
\label{eq;Phik}
	\Phi^k\colon X^k= S\times\C^n \hra  X^{k+1}=S\times\C^n,\quad 
	\Phi^k(s,z)=\bigl(s,\phi^k_s(z)\bigr)
\end{equation}
where $s\in S,\ z\in \C^n$, and $k=1,2,\ldots$.
Set $\Omega^k=\Phi^k(X^k)\subset X^{k+1}$. 
Thus for every fixed $k\in\N$ and $s\in S$ 
the map $\phi^k_s\colon \C^n\hra \C^n$ is biholomorphic 
onto its image $\phi^k_s(\C^n)=\Omega^k_s\subset \C^n$
and it depends holomorphically on the parameter $s\in S$.
In particular, if $\Omega^k_s$ is a proper subdomain 
of $\C^n$ then $\phi^k_s$ is a {\em Fatou-Bieberbach map}.
Let $X$ be the disjoint union of all $X^k$ for $k\in\N$
modulo the following equivalence relation.
A point $x\in X^i$ is equivalent to a point $x'\in X^k$ 
if and only if one of the following hold:
\begin{itemize}
\item[\rm (a)] $i=k$ and $x=x'$,
\item[\rm (b)] $k>i$ and $\Phi^{k-1}\circ\cdots \circ\Phi^i(x)=x'$, or
\item[\rm (c)] $i>k$ and $\Phi^{i-1}\circ\cdots \circ\Phi^{k}(x')=x$.
\end{itemize} 
For each $k\in\N$ we have an injective map 
$\Psi^{k}\colon X^k \hra  X$ onto the subset 
$\wt X^k=\Psi^k(X^k) \subset X$ which sends any point 
$x\in X^k$ to its equivalence class $[x]\in X$. 
Denoting by $\iota^k\colon \wt X^k\hra\wt X^{k+1}$
the inclusion map, we have  
\begin{equation}
\label{comm}
	\iota^k \circ \Psi^k = \Psi^{k+1}\circ \Phi^k,\quad k=1,2,\ldots. 
\end{equation}
The inverse maps 
$(\Psi^k)^{-1} \colon \wt X^k \bihol X^k = S\times \C^n$
provide local charts on $X$. It is easily verified that 
this endows $X$ with the structure of a 
Hausdorff, second countable complex manifold. 
Since each of the maps $\Phi^k$ respects the fibers over $S$, 
we also get a natural projection $\pi\colon X\to S$ 
which is clearly a submersion. For every $s\in S$ the fiber 
$X_s$ is the increasing union of open subsets 
$\wt X^k_s$ biholomorphic to $\C^n$.
Observe that we get the same limit manifold $X$
by starting with any term of the sequence (\ref{eq;Phik}).

The next lemma follows from the And\'ersen-Lempert theory
\cite{AL}; c.f.\ \cite[Theorem 1.2]{Wold2010}.

\begin{lemma}
\label{lemma1}
Let $\pi\colon X\to S$ be as above. Assume that for some $s\in S$
there exists an integer $k_s\in \N$ such that 
for every $k\ge k_s$ the domain $\Omega^k_s=\phi_s^k(\C^n)\subset \C^n$ is 
Runge in $\C^n$. Then $X_s$ is biholomorphic to $\C^n$.
\end{lemma}

\begin{proof}
%
%
The main point is that any biholomorphic map $\C^n\bihol \Omega$
onto a Runge domain $\Omega\subset \C^n$ can be approximated
uniformly on compacts  by holomorphic automorphisms of $\C^n$.
This observation allows one to renormalize the sequence of
biholomorphisms $(\Psi^k_s)^{-1} \colon \wt X^k_s \bihol \C^n$ 
for $k\ge k_s$ so that the new sequence converges uniformly on 
compact in $X_s$ to a biholomorphic map $X_s \bihol \C^n$;
we leave out the straightforward details.
\end{proof}

\section{Entire families of holomorphic automorphisms}
\label{S3}
Let $\aleph_{\cO}(X)$ denote the complex Lie algebra of all
holomorphic vector fields on a complex manifold $X$.

A vector field $V\in \aleph_{\cO}(X)$ is said to be 
{\em $\C$-complete}, or {\em completely integrable},
if its flow $\{\phi_t\}_{t\in \C}$ exists for all complex
values $t\in \C$, starting at an arbitrary point  
$x\in X$. Thus $\{\phi_t\}_{t\in \C}$ is a complex 
one-parameter subgroup of the holomorphic automorphism
group $\Aut X$. The manifold $X$ 
is said to enjoy the (holomorphic) {\em density property} if the Lie 
subalgebra $\Lie(X)$ of $\aleph_{\cO}(X)$ generated 
by the $\C$-complete holomorphic vector fields 
is dense in $\aleph_{\cO}(X)$ in the topology of 
uniform convergence on compacts in $X$ 
(see Varolin \cite{Varolin2001,Varolin2000}).
More generally, a complex Lie subalgebra $\ggoth$ of $\aleph_{\cO}(X)$
enjoys the density property if $\ggoth$ is densely
generated by the $\C$-complete vector fields that it contains.
This property is very restrictive on open manifolds. 
The main result of the And\'ersen-Lempert theory \cite{AL}
is that $\C^n$ for $n>1$ enjoys the density property; in fact,
every polynomial vector field on $\C^n$ is a finite sum of complete
polynomial vector fields (the shear fields).

Varolin proved \cite{Varolin2001} that any domain of the
form $(\C^*)^k\times \C^l$ with $k+l\ge 2$ and $k\ge 1$
enjoys the density property; we shall need this for 
the manifold $\C^*\times \C$. (Here $\C^*=\C\bs \{0\}$.)

\begin{lemma}
\label{lemma2}
Assume that $X$ is a Stein manifold with the density property.
Choose a distance function $\dist_X$ on $X$.
Let $\psi_1, \ldots,\psi_k\in \Aut X$ be 
such that for each $j=1,\ldots, k$ there exists a 
$\cC^2$ path $\theta_{j,t}\in \Aut X$ $(t\in [0,1])$
with $\theta_{j,0}=\Id_X$ and $\theta_{j,1}=\psi_j$.
Given distinct points $a_1,\ldots,a_k\in \C^*$, 
a compact set $K\subset X$ and a number $\epsilon>0$,
there exists a holomorphic map $\Psi \colon \C\times X\to X$
satisfying the following properties:
\begin{itemize}
\item[\rm (i)]   $\Psi_\zeta=\Psi(\zeta,\cdotp)\in \Aut X$ for all $\zeta\in \C$, 
\item[\rm (ii)]  $\Psi_0=\Id_X$,
\item[\rm (iii)] $\sup_{x\in K} \dist_X\bigl(\Psi(a_j,x),\psi_j(x)\bigr) <\epsilon$
for $j=1,\ldots,k$.
\end{itemize}
\end{lemma}

A holomorphic map $\Psi$ satisfying property (i) will be called
an {\em entire curve of holomorphic automorphisms of $X$.}
Here $\Id_X$ denotes the identity on $X$.

\begin{proof}
Consider a $\cC^2$ path $[0,1]\ni t\mapsto \gamma_t\in \Aut X$.
Pick a Stein Runge domain $U\subset X$ containing the set $K$.
Then $U_t=\gamma_t(U) \subset X$ is Runge in $X$ for all $t\in [0,1]$.
By \cite{AL} or, more explicitly, by (the proof of) 
\cite[Theorem 1.1]{FRosay} there exist finitely many complete 
holomorphic vector fields $V_1,\ldots,V_m$ on $X$,
with flows $\theta_{j,t}$, and numbers $c_1>0,\ldots,c_m>0$
such that the composition  
$\theta_{m,c_m}\circ\ldots \circ \theta_{1,c_1} \in \Aut X$
approximates the automorphism $\psi=\gamma_1$ 
within $\epsilon$ on the set $K$. 
(The proof in \cite{FRosay} is written for $X=\C^n$,
but it applies in the general case stated here. We first approximate 
$\gamma_t\colon U\to U_t$ by compositions of short time flows of globally 
defined holomorphic vector fields on $X$; here we need the Runge
property of the sets $U_t$. Since $X$ enjoys the density
property, these vector fields can be approximated 
by Lie combinations (using sums and commutators) of complete vector fields. 
This approximates $\gamma_t$ for each $t\in [0,1]$, uniformly on $K$, 
by compositions of flows of complete holomorphic vector fields on $X$.)
 
Consider $t^1=(t_1,\ldots,t_m)$ as complex coordinates
on $\C^m$. The map
\[
	\C^m\ni (t_1,\ldots,t_m) \mapsto \Theta_1(t_1,\ldots,t_m) =
	\theta_{m,t_m}\circ\ldots \circ \theta_{1,t_1} \in \Aut X
\]
is entire, its value at the origin $0\in \C^m$ is $\Id_X$,
and its value at the point $(c_1,\ldots,c_m)$ is an automorphism
that is $\epsilon$-close to $\psi=\gamma_1$ on $K$.

Using this argument we find for every $j=1,\ldots,k$
an integer $m_j\in \N$ and an entire map 
$\Theta_j\colon \C^{m_j} \to \Aut X$ such that 
$\Theta_j(0)=\Id_X$ and $\Theta_j(c^j_1,\ldots,c^j_{m_j})$ is 
$\epsilon$-close to $\psi_j$ on $K$ at some point 
of $c^j=(c^j_1,\ldots,c^j_{m_j})\in \C^{m_j}$.
Let $t=(t^1,\ldots,t^k)$ be the complex coordinates 
on $\C^M=\C^{m_1}\oplus\cdots \oplus \C^{m_k}$, where
$t^j=(t_1^j,\ldots,t^j_{m_j})\in \C^{m_j}$. The composition
\[
	\C^M\ni t\mapsto 
	\Theta(t^1,\ldots,t^k) = \Theta^k(t^k)\circ \cdots\circ \Theta^1(t^1) \in \Aut X
\]
is an entire map satisfying $\Theta(0)=\Id_X$ such that 
$\Theta(0,\ldots,0,c^j,0,\ldots,0)$ is $\epsilon$-close to 
$\psi_j$ on $K$ for each $j=1,\ldots,k$.

Choose an entire map $g\colon \C\to \C^M$ with
$g(a_j)= (0,\ldots,c^j,\ldots,0)$ for $j=1,\ldots,k$
and $g(0)=0$. Then the map 
$\C\ni \zeta \mapsto \Psi(\zeta)=\Theta(g(\zeta))\in \Aut X$
satisfies the conclusion of the lemma.
\end{proof}

\section{Proof of Theorem \ref{t:main}}
\label{S4}
We shall need the following result from \cite[\S 2]{Wold2008}.
This construction is due to Stolzenberg \cite{Stolz}; 
see also \cite[pp.\ 392--396]{Stout}.

\begin{lemma}
\label{lemma-W2008}
There exists a compact set $Y\subset\C^*\times \C$ (a union $Y=D_1\cup D_2$ 
of two embedded disjoint polynomially convex discs) such that 
\begin{itemize}
\item[\rm (i)] $Y$ is $\cO(\C^*\times\C)$-convex,
\item[\rm (ii)] 
the  polynomial hull $\widehat Y$ contains the origin $(0,0)\in \C^2$, and 
\item[\rm (iii)] 
for any nonempty open set $U\subset\C^*\times\C$ 
there exists a holomorphic automorphism 
$\psi\in\Aut (\C^*\times\C)$ such that 
$Y\subset \psi(U)$.
\end{itemize}
\end{lemma}

Property (iii) is \cite[Lemma 3.1]{Wold2008}:
Since $\C^*\times\C$ enjoys the density property
according to Varolin \cite{Varolin2001}, the isotopy that shrinks 
each of the two discs $D_1,D_2\subset Y$ 
to a point in $U$ can be approximated by an isotopy 
of automorphisms of $\C^*\times\C$ by using the methods
in \cite{FRosay}.

\medskip\noindent
{\em Proof of Theorem \ref{t:main}.}
We give the proof for $n=2$.
Let $B=\{b_1,b_2,\ldots\}$ be as in the theorem.
Choose a set $Y\subset \C^*\times \C$ satisfying Lemma \ref{lemma-W2008}.
Pick a closed ball $K\subset \C^2$ (or any compact set with nonempty interior).

We shall inductively construct a sequence of injective holomorphic maps 
$\Phi^k\colon S\times\C^2 \hra  S\times\C^2$ $(k=1,2,\ldots)$ of the form
\[
	\Phi^k(s,z)=(s,\phi^k_s(z)),\quad s\in S,\ z\in \C^2,
\]
such that, setting 
\begin{equation}
\label{phiks}
	\wt \phi^k_s=\phi^k_s\circ\phi^{k-1}_s\circ\cdots\circ\phi^1_s,
	\qquad 
	K^{k}_s = \wt \phi^k_s (K) \subset \C^2,
\end{equation}
the following properties hold for all $k\in\N$:
\begin{itemize}
\item[\rm (i)]   
$\Omega^{k}:=\Phi^k(S\times\C^2)\subset S\times (\C^*\times\C)$,
\item[\rm (ii)]  the fiber $\Omega^k_s =\phi^k_s(\C^2)$ 
is Runge in $\C^2$ for all $s\in A_1\cup\cdots\cup A_k$, and
\item[\rm (iii)] $Y\subset {\rm Int}\, K^{k}_s$ for each 
$s\in \{b_1,\ldots,b_k\}$. In particular, the polynomial hull
of the set $K^{k}_s$ contains the origin for every such $s$. 
\end{itemize}

Suppose for the moment that we have such a sequence.
Let $X$ denote the limit manifold and let
$\Psi^k\colon X^k=S\times\C^2\bihol \wt X^k \subset X$ be the 
induced inclusions (see \S\ref{S2}). 

If $s\in \cup_k A_k=A$ then property (ii) insures, 
in view of Lemma \ref{lemma1}, that the 
fiber $X_s$ is biholomorphic to $\C^2$. 

Suppose now that $s=b_j$ for some $j\in \N$. Property (iii)
shows that for every integer $k\ge j$ the polynomial 
hull of the set $K^{k}_s$ contains the origin of $\C^2$; 
in particular, $\wh{K^{k}_s}$ is not contained in 
$\Omega^k_s\subset \C^*\times\C$. 
For the corresponding subsets of the limit manifold 
$X_s$ we get in view of (\ref{comm}) that 
\[
		\wh{\Psi^{k+1}_s(K^k_s)} \not\subset \wt X^k_s, \qquad k=j,j+1,\ldots,
\]
where the hull is with respect to the 
algebra of holomorphic functions on the domain 
$\wt X^{k+1}_s \subset X_s$. Let $K_s=\Psi^1_s(K)$ 
denote the compact set in the fiber $X_s$ determined by $K$;
note that $K_s\subset \wt X^1_s$ and 
$K_s=\Psi^{k+1}_s(K^k_s)$ for any $k\in\N$
according to (\ref{comm}) and (\ref{phiks}). 
The above display then gives
\[
	\wh{(K_s)}_{\cO(\wt X^{k+1}_s)} \not\subset \wt X^k_s,
	\quad k=1,2,\ldots.
\]
Since $\wt X^{k+1}_s$ is a domain in $X_s$, we trivially have 
$\wh{(K_s)}_{\cO(\wt X^{k+1}_s)} \subset \wh{(K_s)}_{\cO(X_s)}$; 
hence the hull $\wh{(K_s)}_{\cO(X_s)}$ 
is not contained in $\wt X^k_s$ for any $k\in \N$. As the 
domains $\wt X^k_s$ exhaust $X_s$, this hull is noncompact. 
Hence $X_s$ is not holomorphically convex 
(and therefore not Stein) for any $s \in B$. 

This proves Theorem \ref{t:main} provided that we can find a sequence
with the stated properties.

We begin with some initial choices of domains and maps.
Pick a Fatou-Bieberbach map $\theta\colon \C^2\bihol D \subset \C^*\times\C$
whose image $D=\theta(\C^2)$ is Runge in $\C^2$. (Such $D$ can be obtained 
as the attracting basin of an automorphism of $\C^2$ 
that fixes the complex line $\{0\}\times \C$.) 
Let $U={\rm Int}\, K\subset \C^2$; then 
$\theta(U)\subset D$ is a nonempty open set in $\C^*\times\C$.
For each $k=1,2,\ldots$ we choose a holomorphic function 
$f_k\colon S\to \C$ such that $f_k=0$ on the subvariety 
$A_1\cup\cdots\cup A_k$ of $S$ and $f_k(b_j)=j$ for 
$j=1,\ldots,k$. (If the set $B\subset X\bs A$ also contains a closed complex
subvariety $B'$ of $X$ of positive dimension, we let $f_k=1$ on $B'$.)

We now construct the first map $\Phi^1(s,z)=(s,\phi^1_s(z))$.
Lemma \ref{lemma-W2008} furnishes an automorphism
$\psi \in\Aut (\C^*\times\C)$ such that $Y\subset \psi(\theta(U))$.
By Lemma \ref{lemma2} there exists an entire curve
of automorphisms $\Psi_\zeta \in \Aut (\C^*\times \C)$
$(\zeta\in\C)$ such that $\Psi_0=\Id_{\C^*\times \C}$
and $\Psi_1$ approximates $\psi$ close enough on the 
compact set $\theta(K)$ 
so that $Y\subset \Psi_1(\theta(U))$. Hence 
$(0,0) \in \wh Y \subset \wh{\Psi_1(\theta(K))}$. Set 
\[
	\phi^1_s(z) = \Psi_{f_1(s)}(\theta(z)), \quad s\in S,\ z\in \C^2.
\]
If $s\in A_1$ then $f_1(s)=0$ and hence 
$\phi^1_s(z)=\Psi_0(\theta(z))=\theta(z)$,
so $\phi^1_s=\theta$. If $s=b_1$ then $f_1(s)=1$ and 
hence $\phi^1_s=\Psi_1\circ \theta$. Thus 
$Y\subset \phi^1_{b_1}(U)$ and the polynomial hull 
$\wh{\phi^1_{b_1}(K)}$ contains the origin of $\C^2$. 
This gives the initial step. 

Suppose that we have found maps $\Phi^1,\ldots,\Phi^k$
satisfying conditions (i)--(iii) above; we now construct 
the next map $\Phi^{k+1}$ in the sequence. 
Recall that $\wt\phi^k_s\colon \C^2\to \C^2$ 
is the map defined by (\ref{phiks}). Set  
\[
	U^k_s = (\theta \circ \wt\phi^k_s) (U), \quad s\in S;
\]	
this is a nonempty open set contained in the compact set 
$\theta(K^k_s) \subset \C^*\times \C$.
Lemma \ref{lemma-W2008} gives for each $j=1,\ldots,k+1$
an automorphism $\psi_j\in \Aut (\C^*\times\C)$ 
such that $Y\subset \psi_j(U^k_{b_j})$.
By Lemma \ref{lemma2} there exists an entire curve
of automorphisms $\Psi_\zeta \in \Aut (\C^*\times \C)$
$(\zeta\in\C)$ such that $\Psi_0=\Id_{\C^*\times \C}$ and 
$\Psi_j$ approximates $\psi_j$ for every $j=1,\ldots,k+1$. 
If the approximation is close enough on the compact set 
$\theta(K^k_{b_j})$ then $Y\subset (\Psi_j\circ \theta)(K^k_{b_j})$ 
and hence the origin $(0,0)\in\C^2$ is contained  
in the polynomial hull of $(\Psi_j\circ \theta)(K^k_{b_j})$. 
Set 
\[
	\phi^{k+1}_s(z) = \Psi_{f_{k+1}(s)}\circ\theta(z), 
	\quad s\in S,\ z\in \C^2.
\]
If $s\in A_1\cup\cdots\cup A_{k+1}$ then $f_{k+1}(s)=0$
and hence $\phi^{k+1}_s=\theta$.
If $s=b_j$ for some $j=1,\ldots,k+1$ then 
$f_{k+1}(b_j)=j$ and hence $\phi^{k+1}_{b_j}=\Psi_j\circ\theta$;
therefore the polynomial hull of the set 
$\phi^{k+1}_{b_j}(K^k_{b_j})$ contains the origin. 
Taking $\phi^{k+1}_s$ as the next map in the sequence and
setting 
\[
	\wt \phi^{k+1}_s = \phi^{k+1}_s \circ\wt \phi^{k}_s,
	\qquad 
	K^{k+1}_{s}=\phi^{k+1}_s(K^k_s)
\]
we see that properties (i)--(iii) hold also for $k+1$. 
The induction may continue.
This completes the proof of Theorem \ref{t:main}.

\bibliographystyle{amsplain}

\end{document}